\font \fiverm=cmr5

\documentclass[11pt]{amsart}
\usepackage{amsfonts}
\usepackage{amsmath}
\usepackage{amssymb}
\usepackage{color}
\usepackage{graphicx}
\usepackage{xspace}
\usepackage{axodraw-dm}
\input xy
\xyoption{all}
\font \eightrm=cmr8
\def\diagramme #1{\vskip 4mm \centerline {#1} \vskip 4mm}

 \newcommand{\nc}{\newcommand}

\nc{\surj}{\to\hskip -3mm \to}

\newtheorem{thm}{Theorem}

\newtheorem{prop}[thm]{Proposition}

\def\racine{{\scalebox{0.3}{
\begin{picture}(12,12)(38,-38)
\SetWidth{0.5} \SetColor{Black} \Vertex(45,-30){6}
\end{picture}
}}}

 \def\triangle{\,{\scalebox{0.12}{ 
  \begin{picture}(163,272) (-223,168)
    \SetWidth{3.0}
    \SetColor{Black}
    \Vertex(-213,180){15}
    \Vertex(-71,179){15}
    \Vertex(-138,430){15}
    \ArrowLine(-210,180)(-77,180)
    \ArrowLine(-72,182)(-137,425)
    \ArrowLine(-214,183)(-141,425)
  \end{picture}
}}\,}

\def\graphcograph{\,{\scalebox{0.3}{ 
 \begin{picture}(614,503) (28,-43)
    \SetWidth{1}
    \SetColor{Black}
    \CBox(73.54,390.35)(114.09,436.56){Black}{White}
    \CBox(480.87,398.84)(550.64,449.76){Black}{White}
    \CBox(506.33,240.44)(556.3,289.47){Black}{White}
    \CBox(268.72,239.49)(308.32,296.07){Black}{White}
    \CBox(320.58,266.84)(452.58,398.84){Black}{White}
    \CBox(128.23,247.04)(251.75,416.75){Black}{White}
    \Vertex(95.23,412.98){14.27}
    \Vertex(165,310.21){14.27}
    \Vertex(165.95,374.32){14.27}
    \Vertex(224.41,339.44){14.27}
    \Vertex(346.98,303.61){14.27}
    \Vertex(347.92,373.38){14.27}
    \Vertex(421.47,301.72){14.27}
    \Vertex(420.53,372.44){14.27}
    \Vertex(518.59,423.35){14.27}
    \Vertex(524.24,263.06){14.27}
    \Vertex(284.75,264.95){14.27}
    \SetWidth{2.0}
    \ArrowLine(168.78,378.1)(221.58,341.32)
    \ArrowLine(161.23,372.44)(161.23,313.98)
    \ArrowLine(167.83,310.21)(221.58,340.38)
    \ArrowLine(349.81,376.21)(349.81,308.32)
    \ArrowLine(353.58,373.38)(415.81,373.38)
    \ArrowLine(422.41,372.44)(421.47,306.44)
    \ArrowLine(352.64,303.61)(418.64,304.55)
    \SetWidth{1}
    \ArrowLine(85.8,418.64)(155.58,379.98)
    \ArrowLine(231.95,340.38)(352.64,376.21)
    \ArrowLine(345.1,303.61)(284.75,263.06)
    \ArrowLine(293.24,257.41)(230.06,336.61)
    \ArrowLine(33,460.13)(96.17,412.04)
    \ArrowLine(97.12,422.41)(516.7,426.18)
    \ArrowLine(423.35,299.84)(519.53,267.78)
    \ArrowLine(514.81,423.35)(521.41,268.72)
    \ArrowLine(530.84,264.01)(640.22,308.32)
    \ArrowLine(524.24,423.35)(643.05,431.84)
    \ArrowLine(283.81,207.43)(286.64,259.29)
    \Vertex(280.04,48.09){14.27}
    \Vertex(516.7,152.75){14.27}
    \Vertex(182.92,97.12){17.91}
    \Vertex(91.46,150.86){14.27}
    \Vertex(393.18,101.83){17.91}
    \Vertex(528.96,34.89){14.27}
    \ArrowLine(28.29,192.35)(83.92,155.58)
    \ArrowLine(93.35,149.92)(178.2,99)
    \ArrowLine(94.29,152.75)(511.04,153.69)
    \ArrowLine(516.7,150.86)(528.96,36.77)
    \ArrowLine(518.59,154.63)(624.19,165)
    \ArrowLine(534.61,33.94)(620.42,71.66)
    \ArrowLine(396.01,101.83)(523.3,34.89)
    \ArrowLine(388.47,99)(282.87,48.09)
    \ArrowLine(277.21,-8.49)(280.98,44.32)
    \ArrowLine(280.04,52.8)(181.98,94.29)
    \ArrowLine(189.52,99.95)(382.81,107.49)
    \Text(33,-43.37)[lb]{\huge{\Black{A graph $\Gamma$ together with a covering subgraph $\gamma$ and the contracted graph $\Gamma/\gamma$.}}}
    \SetWidth{2.0}
    \Line(159.35,377.15)(128.23,394.13)
    \Line(234.78,341.32)(251.75,346.04)
    \Line(231.95,331.89)(251.75,309.27)
    \Line(339.44,372.44)(321.52,366.78)
    \Line(340.38,300.78)(320.58,285.69)
    \Line(429.01,297.01)(452.58,289.47)
  \end{picture}
  }}\,}
  
\def\hexagone{\,{\scalebox{0.2}{ 
\begin{picture}(534,499) (88,-58)
    \SetWidth{2.0}
    \SetColor{Black}
    \CBox(518,35)(622,355){Black}{White}
    \CBox(88,-58)(427,123){Black}{White}
    \CBox(88,252)(424,441){Black}{White}
    \ArrowLine(166,295)(350,376)
    \SetWidth{1}
    \Vertex(160,293){20.62}
    \Vertex(573,293){20.62}
    \Vertex(576,87){20.62}
    \Vertex(157,82){20.62}
    \Vertex(357,382){20.62}
    \Vertex(362,-16){20.62}
    \SetWidth{2.0}
    \ArrowLine(571,293)(366,378)
    \ArrowLine(576,294)(576,94)
    \ArrowLine(363,-16)(572,82)
    \ArrowLine(359,-14)(163,82)
    \ArrowLine(160,290)(159,89)
  \end{picture}
  }}\,}

 \def\boucledeux{\,{\scalebox{0.12}{ 
 \begin{picture}(400,400) (300,0)
    \SetWidth{3.0}
    \SetColor{Black}
    \Vertex(378,10){15}
    \Vertex(380,172){15}
    \ArrowArc(348.17,90.14)(85.97,-65.37,68.27)
    \ArrowArc(415.57,90.24)(88.6,115.81,244.91)
  \end{picture}
}}\hskip -8mm}

 \def\flotdeux{\,{\scalebox{0.12}{ 
   \begin{picture}(67,175) (315,-227)
    \SetWidth{3.0}
    \SetColor{Black}
    \Vertex(372,-217){15}
    \Vertex(372,-63){15}
    \ArrowLine(371,-65)(371,-214)
    \ArrowArc(398.67,-141)(81.82,109.76,250.24)
  \end{picture}
}}\,}

 \def\propag{\,{\scalebox{0.12}{ 
   \begin{picture}(67,175) (315,-227)
    \SetWidth{3.0}
    \SetColor{Black}
    \Vertex(372,-217){15}
    \Vertex(372,-63){15}
    \ArrowLine(371,-65)(371,-214)
  \end{picture}
}}\,}

\def\source{\,{\scalebox{0.12}{ 
  \begin{picture}(30,412) (0,215)
    \SetWidth{3.0}
    \SetColor{Black}
    \Vertex(10,315){15}
    \ArrowLine(10,326)(10,412)
    \ArrowLine(10,307)(10,223)
  \end{picture}
}}\,}

 \def\flot{\,{\scalebox{0.12}{ 
   \begin{picture}(30,412) (0,215)
    \SetWidth{3.0}
    \SetColor{Black}
    \Vertex(10,315){15}
    \ArrowLine(10,225)(10,311)
    \ArrowLine(10,326)(10,412)
  \end{picture}
}}\,}

 \def\puits{\,{\scalebox{0.12}{ 
  \begin{picture}(30,401) (0,215)
    \SetWidth{3.0}
    \SetColor{Black}
    \Vertex(10,315){15}
    \ArrowLine(10,225)(10,311)
    \ArrowLine(11,401)(11,315)
  \end{picture}  
}}\,}

\def\sourcebis{\,{\scalebox{0.12}{ 
  \begin{picture}(30,412) (0,215)
    \SetWidth{3.0}
    \SetColor{Black}
    \Vertex(10,315){15}
    \ArrowLine(10,400)(10,480)
    \Vertex(10,400){15}
    \ArrowLine(10,326)(10,412)
    \ArrowLine(10,307)(10,223)
  \end{picture}
}}\,}

 \def\flotbis{\,{\scalebox{0.12}{
 \begin{picture}(30,412) (0,215)
    \SetWidth{3.0}
    \SetColor{Black}
    \Vertex(10,315){15}
    \ArrowLine(10,400)(10,480)
    \Vertex(10,400){15}
    \ArrowLine(10,326)(10,412)
    \ArrowLine(10,223)(10,310)
  \end{picture}
}}\,}

 \def\puitsbis{\,{\scalebox{0.12}{ 
  \begin{picture}(30,412) (0,215)
    \SetWidth{3.0}
    \SetColor{Black}
    \Vertex(10,315){15}
    \ArrowLine(10,480)(10,400)
    \Vertex(10,400){15}
    \ArrowLine(10,326)(10,412)
    \ArrowLine(10,223)(10,310)
  \end{picture}
}}\,}

\def\point{{\scalebox{0.12}{
\begin{picture}(12,12)(-38,38)
\SetWidth{3.0} \SetColor{Black} \Vertex(-30,45){15}
\end{picture}
}}\,}

\nc{\ignore}[1]{{}}
\nc{\mrm}[1]{{\rm #1}}
\nc{\dirlim}{\displaystyle{\lim_{\longrightarrow}}\,}
\nc{\invlim}{\displaystyle{\lim_{\longleftarrow}}\,}
\nc{\vep}{\varepsilon} \nc{\ep}{\epsilon}
\nc{\sigmat}{\widetilde\sigma}
\nc{\ostar}{\overline{*}}

\nc{\mchar}{\mrm{Char}}
\nc{\Hom}{\mrm{Hom}}
\nc{\id}{\mrm{id}}

\nc{\remark}{\noindent{\bf{Remark:}}}
\nc{\remarks}{\noindent{\bf{Remarks:}}}

 \nc{\delete}[1]{}
 \nc{\grad}[1]{^{({#1})}}
 \nc{\fil}[1]{_{#1}}

\nc{\BA}{{\Bbb A}} \nc{\CC}{{\Bbb C}} \nc{\DD}{{\Bbb D}}
\nc{\EE}{{\Bbb E}} \nc{\FF}{{\Bbb F}} \nc{\GG}{{\Bbb G}}
\nc{\HH}{{\Bbb H}} \nc{\LL}{{\Bbb L}} \nc{\NN}{{\Bbb N}}
\nc{\PP}{{\Bbb P}} \nc{\QQ}{{\Bbb Q}} \nc{\RR}{{\Bbb R}}
\nc{\TT}{{\Bbb T}} \nc{\VV}{{\Bbb V}} \nc{\ZZ}{{\Bbb Z}}
\nc{\Cal}[1]{{\mathcal {#1}}}
\nc{\mop}[1]{\mathop{\hbox {\rm #1} }\nolimits}
\nc{\smop}[1]{\mathop{\hbox {\eightrm #1} }\nolimits}
\nc{\ssmop}[1]{\mathop{\hbox {\fiverm #1} }\nolimits}
\nc{\mopl}[1]{\mathop{\hbox {\rm #1} }\limits}
\nc{\frakg}{{\frak g}}
\nc{\g}[1]{{\frak {#1}}}
\def \restr#1{\mathstrut_{\textstyle |}\raise-8pt\hbox{$\scriptstyle #1$}}
\def \srestr#1{\mathstrut_{\scriptstyle |}\hbox to
  -1.5pt{}\raise-4pt\hbox{$\scriptscriptstyle #1$}}
\nc{\wt}{\widetilde}
\nc{\wh}{\widehat}
\nc{\un}{\hbox{\bf 1}}
\nc{\redtext}[1]{\textcolor{red}{\tt #1}}
\nc{\bluetext}[1]{\textcolor{blue}{#1}}
\nc{\comment}[1]{[[{\tt {#1}}]] }
\nc{\R}{\mathbb R}
\nc\fleche[1]{\mathop{\hbox to #1 mm{\rightarrowfill}}\limits}
\def\semi{\mathrel{\times}\kern -.85pt\joinrel\mathrel{\raise
    1.4pt\hbox{${\scriptscriptstyle |}$}}}
\nc{\np}{/\hskip -2.3mm\pi}
\nc{\snp}{/\hskip -1.8mm\pi}

\def\ta1{{\scalebox{0.2}{ 
\begin{picture}(12,12)(38,-38)
\SetWidth{0.5} \SetColor{Black} \Vertex(45,-33){5.66}
\end{picture}}}}
 
\begin{document}

\title[Hopf algebras of oriented graphs]
      {On bialgebras and Hopf algebras of oriented graphs}

\author{Dominique Manchon}
\address{Univ. Blaise Pascal,
         C.N.R.S.-UMR 6620,
         63177 Aubi\`ere, France}       
         \email{manchon@math.univ-bpclermont.fr}
         \urladdr{http://math.univ-bpclermont.fr/~manchon/}

\date{July 4th 2011}
\begin{abstract}
We define two coproducts for cycle-free oriented graphs, thus building up two commutative
connected graded Hopf algebras, such that one is a comodule-coalgebra on the
other, thus generalizing the result obtained in \cite{CEM} for Hopf algebras of
rooted trees.
\end{abstract}

\maketitle


\tableofcontents


\section{Introduction}
\label{sect:intro}
Hopf algebras of graphs have been introduced by D. Kreimer \cite{K98},
\cite{CK0}, \cite{CK1}, \cite{CK2} in order to explain
the combinatorics of renormalization in Quantum Field Theory. Whereas the
product is free commutative, the coproduct is defined by suitable subgraphs
and contracted graphs, and depends on the type of graphs considered.\\

We focus
on various Hopf algebras of oriented graphs: after
giving the basic definitions we detail two examples: oriented graphs in general, and then locally one-particle
irreducible graphs. We show on concrete computations that the coproduct of a
locally 1PI oriented graph $\Gamma$ depends on whether one takes the local
1PI-ness of $\Gamma$ into account or not.\\

We also
explore a third example, the Hopf algebra $\Cal H_{\smop{CF}}$ of oriented \textsl{cycle-free} graphs. The associated poset
structure on the set of vertices yields still another coproduct which generalizes the
coproduct of rooted trees given by admissible cuts \cite{K98}, \cite{F02}. We show that the Hopf
algebra $\Cal H_{\smop{CFc}}$ thus obtained is a comodule-coalgebra on the Hopf
algebra $\Cal H_{\smop{CF}}$. Modulo discarding the external edges, this generalizes
the results of \cite{CEM} on Hopf algebras of rooted trees.
\subsection*{Acknowledgements}
Research partly supported by CNRS, GDR "Renormalisation". I particularly thank Christian Brouder and Fabien Vignes-Tourneret for useful remarks.
\section{Oriented Feynman graphs}
\subsection{Basic definitions}
An {\sl oriented Feynman graph\/} is an oriented (non-planar) graph with a
finite number of vertices and edges, which can be internal or external. An {\sl internal edge\/} is an edge connected at
both ends to a vertex (which can be the same in case of a self-loop), an {\sl
  external edge\/} is an edge with one open end, the other end being connected
to a vertex. An oriented Feynman graph will be called {\sl vacuum graph,
  tadpole graph, self-energy graph\/}, resp. {\sl interaction graph\/} if its
number of external edges is $0$, $1$, $2$, resp. $>2$.\\

A \emph{cycle} in an oriented Feynman graph is a finite collection
$(e_1,\ldots,e_{n})$ of oriented internal edges such that the target of $e_k$
coincides with the source of $e_{k+1}$ for any $k=1,\ldots,n$ modulo $n$. The \emph{loop number} of a graph $\Gamma$ is given by:
\begin{equation}
L(\Gamma)=I(\Gamma)-V(\Gamma)+1,
\end{equation}
where $I(\Gamma)$ is the number of internal edges of the graph $\Gamma$ and where $V(\Gamma)$ is the
number of vertices. We shall mainly focus on cycle-free oriented graphs, for
which there exists a poset structure on the set of vertices: namely, $v<w$ if
and only if there exists a path from $v$ to $w$, i.e. a collection
$(e_1,\ldots,e_n)$ of edges such that the target of $e_k$ coincides with the
source of $e_{k+1}$ for $k=1,\ldots,n-1$, and such that $v$ (resp. $w$) is the
source (resp. the target) of $e_1$ (resp. $e_n$).\\

The edges (internal or external) will be of different types labelled by a
positive integer ($1,2,3,\ldots$), each type being represented by the way the
corresponding edge is drawn (full, dashed, wavy, various colours, etc...). Let
$\tau (e)\in\NN^*$ be the type of the edge $e$. For any vertex $v$ let
$\mop{st}(v)$ be the {\sl star\/} of $v$, i.e. the set of all half-edges attached
to $v$ (hence a self-loop yields two half-edges). Hence the valence of the vertex
is given by the cardinal of $\mop{st}(v)$. Finally to each vertex $v$ we
associate its {\sl type\/} $T(v)$, defined as the sequence $(n_1,\ldots ,n_r)$ of positive
integers where $n_j$ stands for the number of edges of type $j$ in
$\mop{st}(v)$. The orientation does not enter into the definition of the type
of a vertex.\\

A {\sl one-particle irreducible graph\/} (in short, 1PI graph) is a connected
graph which remains connected when we cut any internal edge. A disconnected
graph is said to be {\sl locally 1PI\/} if any of its connected components is
1PI.
\subsection{Connected subgraphs, covering subgraphs and contracted graphs}
Let $\Gamma$ be an oriented Feynman graph, let $\Cal V(\Gamma)$ be the set of
its vertices, and let $P$ be a non-empty subset of $\Cal V(\Gamma)$. The subgraph $\Gamma(P)$ associated to $P$ is defined
as follows: the internal edges of $\Gamma(P)$ are the internal edges of
$\Gamma$ with source and target in $P$, and the external edges are the
external edges of $\Gamma$ with source or target in $P$, as well as the
internal edges of $\Gamma$ with one end in $P$ and the other end outside
$P$. The orientations of the edges of $\gamma_P$ are obviously derived from
their orientation in $\Gamma$. The subgraph $\Gamma(P)$ is connected if and
only if for any $v,w\in P$ one can go from $v$ to $w$ by
following internal edges of $\Gamma$ with both ends in $P$, forwards or
backwards. We set by convention $\Gamma(\emptyset)=\un$, where $\un$ is the
empty graph. For any $Q\subseteq P\subseteq\Cal V(\Gamma)$ we obviously have:
\begin{equation}
\Gamma(P)(Q)=\Gamma(Q).
\end{equation}
A \textsl{covering subgraph} of $\Gamma$ is an oriented Feynman graph $\gamma$
(in general disconnected), given by
a collection $\{\Gamma(P_1),\ldots,\Gamma(P_n)\}$ of connected
subgraphs such that $P_j\cap P_k=\emptyset$ for $j\not =k$, and such that any vertex of $\Gamma$ belongs to
$P_j$ for some (unique) $j\in\{1,\ldots ,n\}$. Covering subgraphs of $\Gamma$
are in one-to-one correspondence with partitions of $\Cal V(\Gamma)$ into
connected subsets, which refine the partition into connected components.
For any covering subgraph $\gamma$, the \textsl{contracted graph} $\Gamma/\gamma$ is defined by shrinking all
connected components of $\gamma$ inside $\Gamma$ onto a
point. 
\begin{equation*}
\graphcograph
\end{equation*}
\vskip 6mm
\noindent The following proposition is straightforward:
\begin{prop}\label{prop:cf}
Let $\Gamma$ be an oriented cycle-free Feynman graph. Let
$\gamma$ be a covering subgraph and let $V=P_1\sqcup\cdots\sqcup P_n$ be the
associated partition of $V$. If $\Gamma/\gamma$
is cycle-free, then $P_j$ is a convex subset of the poset $\Cal V(\Gamma)$ for any $j\in\{1,\ldots ,n\}$.
\end{prop}
\noindent Note that the converse is not true, as shown by the following counterexample:
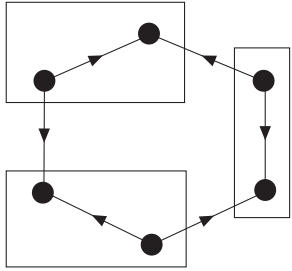
\begin{equation*}
\hexagone
\end{equation*}
The \textsl{residue} of the graph $\Gamma$ is the contracted graph
$\Gamma/\Gamma$, where the covering subgraph is the graph $\Gamma$
itself. The associated partition of $\Cal V(\Gamma)$ is the coarsest possible,
i.e. it is given by its connected components. The residue is the only graph with no internal edge and the same
external edges than those of $\Gamma$. At the other extreme, the partition of
$\Cal V(\Gamma)$ into singletons (i.e. the finest possible) gives rise to the unique covering subgraph $\gamma_0$
without internal edges. The contracted
graph $\Gamma/\gamma_0$ is then equal to $\Gamma$. Given
two covering subgraphs $\gamma$ and $\delta$, say that $\gamma$
\textsl{contains} $\delta$ if the partition associated with $\delta$ refines
the partition associated with $\gamma$. In this case $\delta$ can also be seen
as a covering subgraph of $\gamma$.
\section{Some bialgebras and Hopf algebras of graphs}
\subsection{The full bialgebra of oriented Feynman graphs}\label{sect:ofg}
Let $\wt{\Cal H}$ be the vector space spanned by (connected or not) oriented
Feynman graphs. The product is given by concatenation, hence $\wt{\Cal H}=S(V)$,
where $V$ is the vector space spanned by connected oriented
Feynman graphs. The unit $\un$ is identified with the empty graph, and the coproduct
is given by:
\begin{equation}
\Delta(\Gamma)=\sum_{\gamma \smop{ covering subgraph of }\Gamma}\gamma\otimes\Gamma/\gamma.
\end{equation}
This is obviously an algebra morphism, and we have:
\begin{align*}
(\Delta\otimes I)\Delta(\Gamma)&=
\sum_{\delta\smop{ covering subgraph of }\gamma,\ \gamma\smop{ covering subgraph of }\Gamma} \delta\otimes\gamma/\delta\otimes\Gamma/\gamma,\\
(I\otimes\Delta)\Delta(\Gamma)&=
\sum_{\delta\smop{ covering subgraph of }\Gamma, \ \wt\gamma\smop{ covering subgraph of }\Gamma/\delta} \delta\otimes\wt\gamma\otimes(\Gamma/\delta)/\wt\gamma.
\end{align*}
There is an obvious bijection $\gamma\mapsto\wt\gamma=\gamma/\delta$ from
covering subgraphs of $\Gamma$ containing $\delta$ onto covering subgraphs of
$\Gamma/\delta$, given by shrinking $\delta$. As we have the obvious
``transitive shrinking property'':
\begin{equation}
\Gamma/\gamma=(\Gamma/\delta)/(\gamma/\delta),
\end{equation}
the two expressions coincide, hence $\Delta$ is coassociative. The co-unit is
given by $\varepsilon(\un)=1$ and $\varepsilon(\Gamma)=0$ for any non-empty
graph $\Gamma$. The bialgebra $\wt{\Cal H}$ is graded by the number of internal
edges (and even multi-graded by the numbers of internal edges of various given
types). The elements of degree zero are the residues, i.e. the graphs without
internal edges. Any residue graph $R$ is grouplike, i.e. $\Delta(R)=R\otimes
R$. As an example of coproduct computation (with only one type of edges), we have:
\begin{equation}\label{eq:coproduit1}
\Delta(\triangle)=\source\flot\puits\otimes\triangle+\triangle\otimes\point
+\source\puitsbis\otimes\flotdeux+\flot\flotbis\otimes\boucledeux
+\puits\sourcebis\otimes\flotdeux.
\end{equation}
\subsection{The Hopf algebra of oriented Feynman graphs}\label{sect:hofg}
The Hopf algebra $\Cal H$ is obtained from $\wt{\Cal H}$ by identifying all
degree zero elements with the unit $\un$, namely:
\begin{equation}
\Cal H=\wt{\Cal H}/\Cal J,
\end{equation}
where $\Cal J$ is the (bi-)ideal generated by the elements $\Gamma-\un$ where
$\Gamma$ is any graph without internal edges. The bialgebra $\Cal H$ is
obviously connected graded, hence it is a Hopf algebra, which can be
identified as a commutative algebra with $S(W)$, where $W$ is the vector space
spanned by connected oriented Feynman graphs with at least one internal edge. The coproduct computation \eqref{eq:coproduit1} yields:
\begin{equation}\label{eq:coproduit2}
\Delta(\triangle)=\un\otimes\triangle+\triangle\otimes\un
+\puitsbis\otimes\flotdeux+\flotbis\otimes\boucledeux
+\sourcebis\otimes\flotdeux.
\end{equation}
\subsection{Locally 1PI graphs}\label{sect:1pi}
A similar construction holds for locally 1PI graphs: the bialgebra $\wt{\Cal
  H}_{\smop{1PI}}$ is given by $S(V_{\smop{1PI}})$, where $V_{\smop {1PI}}$ is the vector space
spanned by connected oriented 1PI Feynman graphs. The coproduct is given by:
\begin{equation}
\Delta(\Gamma)=\sum_{\gamma \smop{ locally 1PI covering subgraph of }\Gamma}\gamma\otimes\Gamma/\gamma,
\end{equation}
and is coassociative due to the fact that the transitive shrinking property of
Paragraph \ref{sect:ofg} still makes sense for locally 1PI covering
subgraphs. The (multi-) grading given by the number of internal edges is still
relevant, but an alternative grading is given by the loop number. The
associated Hopf algebra ${\Cal H}_{\smop{1PI}}$ is built up similarly to $\Cal
H$ in Paragraph \ref{sect:hofg}, by identifying the elements of degree zero
with the unit $\un$. Note that, for both gradings, the elements of degree zero
are the residues: it comes from the fact that a graph $\Gamma$ with loop
number $L(\Gamma)=0$ which is locally 1PI cannot have any internal edge. Here is an example of coproduct computation, in $\wt{\Cal H}_{\smop{1PI}}$ and ${\Cal H}_{\smop{1PI}}$ respectively:
\begin{eqnarray}
\Delta(\triangle)&=&\source\flot\puits\otimes\triangle+\triangle\otimes\point,\label{eq:coproduit3}\\
\Delta(\triangle)&=&\un\otimes\triangle+\triangle\otimes\un.\label{eq:coproduit4}
\end{eqnarray}
\subsection{Cycle-free graphs}\label{sect:cf}
Let $\Gamma$ be a cycle-free oriented Feynman graph. In view of Proposition
\ref{prop:cf}, we say that a covering subgraph $\gamma$ of $\Gamma$ is
\textsl{poset-compatible} if the contracted graph $\Gamma/\gamma$ is cycle-free. It implies that all elements of the associated partition are convex
subsets of the poset $\Cal V(\Gamma)$. The bialgebra $\wt{\Cal
  H}_ {\smop{CF}}$ is given by $S(V_{\smop{CF}})$, where $V_{\smop {CF}}$ is the vector space
spanned by connected oriented cycle-free Feynman graphs. The coproduct is given by:
\begin{equation}
\Delta(\Gamma)=\sum_{\gamma \smop{ poset-compatible covering subgraph of }\Gamma}\gamma\otimes\Gamma/\gamma,
\end{equation}
and is coassociative due to the fact that the transitive shrinking property of
Paragraph \ref{sect:ofg} still makes sense for poset-compatible covering
subgraphs of a cycle-free graph. The (multi-) grading given by the number of internal edges is still
relevant, and the
associated Hopf algebra ${\Cal H}_{\smop{CF}}$ is built up similarly to $\Cal
H$ in Paragraph \ref{sect:hofg}, by identifying the elements of degree zero
with the unit $\un$. Note that, contrarily to the previous examples, the
orientation of the edges enters here in an essential way. Our favourite coproduct computation takes the following form, in $\wt{\Cal H}_ {\smop{CF}}$ and ${\Cal H}_{\smop{CF}}$ respectively:
\begin{eqnarray}
\Delta(\triangle)&=&\source\flot\puits\otimes\triangle+\triangle\otimes\point
+\source\puitsbis\otimes\flotdeux
+\puits\sourcebis\otimes\flotdeux,\label{eq:coproduit5}\\
\Delta(\triangle)&=&\un\otimes\triangle+\triangle\otimes\un
+\puitsbis\otimes\flotdeux
+\sourcebis\otimes\flotdeux.\label{eq:coproduit6}
\end{eqnarray}
\subsection{Cycle-free locally 1PI graphs}
We can combine Paragraphs \ref{sect:1pi} and \ref{sect:cf}~: the bialgebra
$\wt{\Cal H}_{\smop{CF1PI}}$ of cycle-free locally 1PI graphs is given by the
intersection $\wt{\Cal H}_{\smop{CF}}\cap\wt{\Cal H}_{\smop{1PI}}$. This is the
free commutative algebra on the vector space spanned by the space
$V_{\smop{CF1PI}}$ of connected cycle-free locally 1PI graphs, and the
coproduct is given by:
\begin{equation}
\Delta(\Gamma)=\sum_{{\gamma \smop{ poset-compatible locally 1PI}\atop \smop{ covering subgraph of }\Gamma}}\gamma\otimes\Gamma/\gamma,
\end{equation}
and the associated Hopf algebra ${\Cal H}_{\smop{CF1PI}}$ is obtained by
identifying the residue graphs with the empty graph $\un$. Details are left to the reader.
\section{A comodule-coalgebra on the bialgebra of oriented cycle-free graphs}
\subsection{Another Hopf algebra structure on oriented cycle-free graphs}\label{AC}
Consider the bialgebra $\wt{\Cal H}_ {\smop{CF}}=S(V_{\smop{CF}})$ of
Paragraph \ref{sect:cf}. We keep the same commutative product, but we define
another coproduct as follows. For any cycle-free oriented graph $\Gamma$ we
set:
\begin{equation}
\Delta_{c}(\Gamma)=\sum_{V_1\sqcup V_2=\Cal V(\Gamma), \ V_2<V_1}\Gamma(V_1)\otimes\Gamma(V_2).
\end{equation}
The inequality $V_2<V_1$ means that for any comparable $v_1\in V_1$ and
$v_2\in V_2$ we have $v_2<v_1$ in the poset $\Cal V(\Gamma)$. Such a pair of
disjoint subsets will be called an \textsl{admissible cut}. It matches the
usual notion of admissible cut when the graph $\Gamma$ is a
rooted tree\cite{K98}, \cite{F02}, \cite{Murua}. Note however that the relation $<$ on the set of subsets of $\Cal V(\Gamma)$ is not transitive. The coproduct is obviously coassociative, as we
have:
\begin{equation}
(I\otimes\Delta_{c})\Delta_{c}(\Gamma)=(\Delta_{c}\otimes
I)\Delta_{c}(\Gamma)=\sum_{V_1\sqcup V_2\sqcup V_3=\Cal V(\Gamma), \ V_3<V_2<V_1}\Gamma(V_1)\otimes\Gamma(V_2)\otimes\Gamma(V_3),
\end{equation}
where the notation $V_3<V_2<V_1$ means $V_3<V_2$, $V_2<V_1$ {\bf and} $V_3<V_1$. This coproduct is also an algebra morphism. We denote by ${\Cal H}_{\smop{CFc}}$ the
connected graded Hopf algebra given by this coproduct. It is naturally
isomorphic to $\wt{\Cal H}_{\smop{CF}}$ as a commutative algebra, but the grading is now given by the
number of vertices. As an example, we have:
\begin{equation}
\Delta_c(\triangle)=\triangle\otimes\un+\un\otimes\triangle+\source\otimes\puitsbis+\sourcebis\otimes\puits.
\end{equation}
\subsection{The comodule-coalgebra structure on ${\Cal H}_{\smop{CFc}}$}
The coproduct $\Delta$ on the bialgebra $\wt{\Cal H}_{\smop{CF}}$ can also be seen as a
left (resp. right) coaction $\Phi:{\Cal H}_{\smop{CFc}}\to\wt{\Cal H}_{\smop{CF}}\otimes{\Cal
  H}_{\smop{CFc}}$, resp. $\Psi:{\Cal H}_{\smop{CFc}}\to{\Cal H}_{\smop{CFc}}\otimes\wt{\Cal
  H}_{\smop{CF}}$.
 \begin{thm}\label{diagramme}
The left coaction map $\Phi$ verifies:
\begin{equation}
(Id_{\wt {\Cal H}_{\smop{CF}}}\otimes\Delta_{c})\circ \Phi=m^{1,3}\circ
(\Phi\otimes\Phi)\circ\Delta_{c},
\end{equation}
i.e. the following diagram commutes:
\diagramme{
\xymatrix{{\mathcal H}_{\smop{CFc}}\ar[rr]^{\Phi}\ar[d]_{\Delta_{c}}
&&\wt{\mathcal H}_{\smop{CF}}\otimes{\mathcal H}_{\smop{CFc}}\ar[dd]^{I\otimes\Delta_{c}}\\
{\mathcal H}_{\smop{CFc}}\otimes{\mathcal H}_{\smop{CFc}}\ar[d]_{\Phi\otimes\Phi}&&\\
\wt{\mathcal H}_{\smop{CF}}\otimes{\mathcal H}_{\smop{CFc}}\otimes\wt{\mathcal H}_{\smop{CF}}\otimes{\mathcal
  H}_{\smop{CFc}}
\ar[rr]_{m^{1,3}}&&\wt{\mathcal H}_{\smop{CF}}\otimes{\mathcal H}_{\smop{CFc}}\otimes{\mathcal H}_{\smop{CFc}}
}
}
where:
\begin{eqnarray*}
m^{1,3}:\wt{\mathcal H}_{\smop{CF}}\otimes\mathcal{H}_{\smop{CFc}}\otimes\wt{\mathcal H}_{\smop{CF}}\otimes\mathcal{H}_{\smop{CFc}}&\longrightarrow&
\wt{\mathcal H}_{\smop{CF}}\otimes\mathcal{H}_{\smop{CFc}}\otimes\mathcal{H}_{\smop{CFc}}\\
 a\otimes b\otimes c\otimes d& \longmapsto &ac\otimes b\otimes d
\end{eqnarray*}
In other words ${\mathcal H}_{\smop{CFc}}$ is a $\wt{\mathcal H}_{\smop{CF}}$-comodule coalgebra,
i.e. a coalgebra in the category of $\wt{\mathcal H}_{\smop{CF}}$-comodules.
\end{thm}
\begin{proof}
This result is a direct generalization of Theorem 8 in \cite{CEM} and is
proved in a similar way: the verification is immediate for the empty graph. We have for any nonempty graph:
\begin{eqnarray*}
(\mop{Id}_{\wt{\Cal H}_{\ssmop{CF}}}\otimes\Delta_{c})\circ\Phi\big(\Gamma)&=&(\mop{Id}_{\wt{\Cal H}_{\ssmop{CF}}}\otimes\Delta_{c})
						\left(\sum_{\gamma\smop{ poset-compatible covering }\atop \smop{
                              subgraph of } \Gamma}\gamma\otimes \Gamma/\gamma\right)\\
	&=&\sum_{\gamma\smop{ poset-compatible covering }\atop \smop{ subgraph of}\Gamma}\
    \sum_{U_1\sqcup U_2=\Cal V(\Gamma/\gamma),\ U_2<U_1}\gamma\otimes (\Gamma/\gamma)(U_1)\otimes (\Gamma/\gamma)(U_2).
\end{eqnarray*}
On the other hand we compute:
\begin{eqnarray*}
m^{1,3}\circ(\Phi\otimes\Phi)\circ\Delta_{c}(\Gamma)&=&
	m^{1,3}\circ(\Phi\otimes\Phi)\left(\sum_{V_1\sqcup V_2=\Cal V(\Gamma),\ V_2<V_1}\Gamma(V_1)\otimes\Gamma(V_2)\right)\\
	&\hskip -60mm = &\hskip -30mm 
	m^{1,3}\left(\sum_{V_1\sqcup V_2=\Cal V(\Gamma),\ V_2<V_1}\
      \sum_{\gamma'\smop{ poset-compatible covering }\atop \smop{ subgraph of }\Gamma(V_1)}\ 
	\sum_{\gamma''\smop{ poset-compatible covering }\atop \smop{ subgraph of }\Gamma(V_2)}
	\hskip -2mm  \gamma'\otimes \Gamma(V_1)/\gamma'\otimes\gamma''\otimes \Gamma(V_2)/\gamma'' \right)\\
	&\hskip -60mm =&\hskip -30mm 
	\sum_{V_1\sqcup V_2=\Cal V(\Gamma),\ V_2<V_1}\ \sum_{\gamma'\smop{ poset-compatible covering }\atop \smop{ subgraph of }\Gamma(V_1)}\ 
	\sum_{\gamma''\smop{ poset-compatible covering }\atop \smop{ subgraph of }\Gamma(V_2)}
	 \hskip -2mm \gamma'\gamma''\otimes \Gamma(V_1)/\gamma'\otimes \Gamma(V_2)/\gamma''\\
	&\hskip -60mm =&\hskip -30mm 
	\sum_{V_1\sqcup V_2=\Cal V(\Gamma),\ V_2<V_1}\ \sum_{\gamma\smop{ poset-compatible
        covering subgraph of }\Gamma\atop \smop{ without any internal edge
        between }V_1 \smop{ and }V_2}
	\hskip -7mm \gamma\otimes \Gamma(V_1)\Big\slash \gamma\cap \Gamma(V_1)\otimes \Gamma(V_2)\Big\slash \gamma\cap \Gamma(V_2)\\
	&\hskip -60mm =&\hskip -30mm 
	\sum_{\gamma\smop{ poset-compatible covering }\atop \smop{ subgraph of }\Gamma}\
    \sum_{U_1\sqcup U_2=\Cal V(\Gamma/\gamma),\ U_2<U_1}\gamma\otimes (\Gamma/\gamma)(U_1)\otimes(\Gamma/\gamma)(U_2),
\end{eqnarray*}
which proves the theorem.
\end{proof}
\section{Discarding external edges}
The whole discussion can be carried out from the beginning, dealing with graphs with only internal edges, leading to a bialgebra $\wt{\Cal K}$ and a Hopf algebra $\Cal K$, as well as to their variants $\wt{\Cal K}_{\smop{1PI}}$,  $\wt{\Cal K}_{\smop{CF}}$, $\wt{\Cal K}_{\smop{CF1PI}}$, ${\Cal K}_{\smop{1PI}}$,  ${\Cal K}_{\smop{CF}}$ and ${\Cal K}_{\smop{CF1PI}}$. The definition of a subgraph remains the same except that we discard the external edges which could appear. As an example we compute on our favourite example the analog of the coproduct $\Delta$ of Sections \ref{sect:ofg} and \ref {sect:hofg} respectively in this new framework:
\begin{eqnarray}
\Delta(\triangle)&=&\racine\racine\racine\otimes\triangle+\triangle\otimes\point
+2\racine\propag\otimes\flotdeux+\racine\propag\otimes\boucledeux\label{deltabig}\\
\Delta(\triangle)&=&\un\otimes\triangle+\triangle\otimes\un
+2\propag\otimes\flotdeux+\propag\otimes\boucledeux\label{deltahopf}.
\end{eqnarray}
The computations for the cycle-free variants are the same except that we discard the last term. Locally 1PI variants are straightforward and left to the reader. The admissible cut coproduct of Section \ref{AC} reads on the "triangle" graph:
\begin{equation}
\Delta_c(\triangle)=\triangle\otimes\un+\un\otimes\triangle+\racine\otimes\propag+\propag\otimes\racine.
\end{equation}
The coproduct $\Delta$ of $\wt{\Cal K}_{CF}$ coincides on rooted forests
with the coproduct of the bialgebra $\wt {\Cal H}$ of \cite{CEM}, and the coproduct
$\Delta_c$ coincides on rooted forests with the Connes-Kreimer coproduct
$\Delta_{\smop{CK}}$. The main difference here is that there are no pre-Lie structures \cite{MS08} associated with these Hopf algebras of Feynman graphs, which are not right-sided in the sense of \cite{LoRo}.\\

It would be interesting to develop objects similar to B-series such that "composition" and "substitution" of these objects are reflected by this construction, thus generalizing \cite{CHV} and \cite{CEM}.

\end{document}